\documentclass[10pt]{amsart}
\usepackage{amsfonts,amsmath,amssymb,amsthm,latexsym,verbatim}
\newtheorem{theorem}{Theorem}[section]
\newtheorem{lemma}[theorem]{Lemma}

\newtheorem{corollary}[theorem]{Corollary}
\theoremstyle{remark}
\newtheorem{remark}[theorem]{Remark}
\theoremstyle{definition}
\newtheorem{definition}[theorem]{Definition}
\newtheorem{example}[theorem]{Example}

\DeclareMathOperator{\coker}{{\mathrm{coker}}}

\DeclareMathOperator{\Hom}{{\mathrm{Hom}}}

\newcommand{\abs}[1]{|#1|}	
\newcommand{\Fq}{{\mathbb {F}_q}} 
\newcommand{\Fstar}{{\mathbb {F}_q^\times}} 
 
\newcommand{\Z}{{\mathbb {Z}}} 
\newcommand{\Q}{{\mathbb {Q}}}

\DeclareMathOperator{\diag}{{\mathrm {diag}}} 	
\DeclareMathOperator{\rank}{{\mathrm {rank}}} 	
\DeclareMathOperator{\Paley}{{\mathrm {Paley}}} 	
\newcommand{\allone}{{\mathbf 1}} 	

\begin{document}
\title[]{The Critical groups of the Peisert graphs $P^*(q)$}

\author{Peter Sin}
\address{Department of Mathematics\\University of Florida\\ P. O. Box 118105\\ Gainesville FL 32611\\ USA}
\date{}
\thanks{This work was partially supported by a grant from the Simons Foundation (\#204181 to Peter Sin)}
\thanks{Part of this work was done during a visit to the Institute for Mathematical Sciences, Natonal University of Singapore as part of the program ``New Directions in Combinatorics.}

\begin{abstract}
The critical group of a finite graph is an abelian group defined by the
Smith normal form of the Laplacian. We determine the 
the critical groups of the Peisert graphs, a certain
family of strongly regular graphs similar to, but different from,
the Paley graphs. It is further shown that
the adjacency matrices of the two graphs defined over a field of order 
$p^2$ with $p\equiv 3\pmod 4$ are similar over the $\ell$-local integers
for every prime $\ell$. Consequently, each such pair of graphs
provides an example where all the corresponding generalized
adjacency matrices are both cospectral and equivalent in the sense of Smith normal form.
\end{abstract}

\maketitle
\section{Introduction} 
Let $\Gamma=(V,E)$ be a finite, simple, undirected and connected graph and let $A$ be the adjacency matrix of $\Gamma$  with respect to some fixed but arbitrary ordering of the vertex set $V$ of $\Gamma$. Let $D$ be the diagonal matrix whose $(i,i)$-entry is the degree of the $i^{\rm th}$ vertex. Then $L=D-A$ is called the {\it Laplacian matrix} of $\Gamma$. The matrices $A$ and $L$ represent endomorphisms (which will also be denoted by $A$ and $L$)  of the free abelian group on $V$. The structure of their cokernels as abelian groups is independent of the above ordering and can be found by
computing the Smith normal forms of the matrices. The cokernel of $A$ is called the {\it Smith group}.
The endomorphism $L$ maps the sum of all vertices to zero, so its cokernel is not a torsion group. The torsion subgroup $K(\Gamma)$ of the cokernel  of $L$ is called the {\it critical group} of $\Gamma$. 
It is known by Kirchhoff's matrix-tree theorem that the order of $K(\Gamma)$ is equal to the number of spanning trees of $\Gamma$.

One source of motivation for the study of the critical group came from physics \cite{DH},
where it was called the {\it sandpile group}.
In graph theory an early author on the critical group was 
Vince \cite{Vince}, who computed  them for wheels and complete bipartite graphs, 
and pointed out that the critical group depends only on the cycle matroid of the graph.
Other papers containing calculations of critical groups for families of graphs include 
Bai \cite{Bai}, Jacobson \cite{Jacobson}, Jacobson-Niedermaier-Reiner \cite{Jacobson-Niedermaier-Reiner}, Ducey-Jalil \cite{Ducey-Jalil} and  Chandler-Sin-Xiang \cite{CSXPaley}. Lorenzini \cite{Lorenzini} has examined the proportion of graphs with 
cyclic critical groups among graphs having critical groups of a particular  order,
while Wood \cite{Wood} has determined the distribution of the critical groups of the Erd\"os-R\'enyi random graphs.

The object of the present paper is to add one more family to the class of computed examples, by applying some of the ideas used for Paley graphs in \cite{CSXPaley} to the Peisert graphs. We shall obtain a complete description of the group structure of the critical groups of the Peisert graphs. However, unlike in \cite{CSXPaley}, we are not able to obtain a neat description of the generating function for the multiplicities of elementary divisors, so in this sense
the results are less satisfactory.

In the final section we study more closely the Peisert graphs and Paley graphs defined over the field of $p^2$ elements. Suppose $A$ is the adjacency matrix of a graph on $n$ vertices and  let $I$ denote the $n\times n$ identity matrix and $J$ the $n\times n$ matrix
whose entries are all equal to 1. Then the 
{\it generalized adjacency matrices} are the matrices
$aA+bI+cJ$ for integers $a$, $b$ and $c$. Among the generalized 
adjacency matrices are the Seidel $(-1,0,1)$ adjacency matrix,
the adjacency matrix of the complementary graph and, in the  case
of a regular graph, the Laplacian and signless Laplacian matrices, 
We show that, when $q=p^2$,  each generalized adjacency matrix of the Peisert graph 
is cospectral with, and has the same Smith normal form as, the corresponding
generalized adjacency matrix of the Paley graph. These properties are derived from the
stronger property that the adjacency matrices are similar by
an invertible matrix over a  ring of algebraic integers. 

\section{Definitions and notation}
\subsection{The Peisert graphs}
Here, we describe the family of graphs $P^*(q)$ constructed in \cite{Peisert}.
Let $q=p^{2t}$, for a prime $p$ with $p\equiv3\pmod 4$, and $t$ a positive integer. 
Let $\beta$ be a primitive element in $\Fq$. In the multiplicative group $\Fstar$, the subgroup $C_0$  of nonzero $4$-th powers has index $4$. Let $C_1$ be the coset $\beta C_0$ and let
$S'=C_0\cup C_1$.
The graph  $P^*(q)$ has vertex set $\Fq$, with two vertices $x$ and $y$  joined by an edge if and only if $x-y\in S'$. As observed in \cite{Peisert}
the isomorphism type of $P^*(q)$ does not depend on the choice of $\beta$.
The graphs $P^*(q)$ and the Paley graphs $\Paley(q)$  (\cite[p.101]{BH})  
are both Cayley graphs on an elementary
abelian group of order $q$ and are cospectral, but not isomorphic except when $q=9$ (\cite[\S6]{Peisert}). 
 
Thus, the graphs $P^*(q)$ form an infinite family of self-complementary
strongly regular graphs (also known as {\it conference graphs}) of non-Paley type.
There are many ways to construct graphs with the same parameters 
that are Cayley graphs on the same group. (See 
\cite{Mullin}, \cite{WQWX}.) The aim of this paper is to compute certain matrix invariants of the graphs $P^*(q)$, in particular their {\it critical groups}.

\subsection{The Smith group and the critical group} 
Let $R$ be a principal ideal domain. Then the matrix
version of the fundamental theorem on finitely generated $R$-modules says that
every $m\times n$ matrix $X$ over $R$ is {\it $R$-equivalent} to its
{\it Smith normal form}. That is to say, there exist an $m\times m$ matrix $P$ 
and an $n\times n$ matrix $Q$, both invertible over $R$, such that
$PXQ=D$, where 
\begin{equation*}
D=\begin{bmatrix} D_1 &0\\0&0
\end{bmatrix}
\end{equation*}
with $D_1=\diag(s_1,s_2,\ldots,s_r)$,  $s_1\mid s_2\mid\cdots\mid s_r$,
and  $r=\rank X$.
If we consider the $R$-module homomorphism $\mu_X:R^n\to R^m$ given by
left multiplication by $X$, then the Smith normal form
describes the decomposition of $\coker(\mu_X)$ , called the {\it Smith group } of $X$, into cyclic $R$-submodules. 
Sometimes, it is convenient to drop the divisibility requirement
and work with other ``diagonal forms'' of $X$, which also determine the Smith group.

Given a finite graph $\Gamma=(V,E)$, with $V$ ordered in 
some way, two important integer matrices are the adjacency matrix
$A$ and the Laplacian matrix $L$. If we take $R=\Z$ then, as already stated in the Introduction,
the {\it Smith group of $\Gamma$} is defined to be the Smith group of $A$
and the {\it critical group of $\Gamma$} is defined to be
the torsion subgroup of the  Smith group  of $L$ . We shall denote
the critical group of $\Gamma$ by $K(\Gamma)$.

\section{The Smith group and the $p'$-torsion of the critical group of $P^*(q)$}\label{eval}
$P^*(q)$ is a  strongly regular graph, cospectral with $\Paley(q)$.
The eigenvalues of its adjacencey matrix are $k=\frac{q-1}{2}$,
$r=\frac{-1+\sqrt{q}}{2}$ and $s=\frac{-1-\sqrt{q}}{2}$, with multiplicities $1$, $\frac{q-1}{2}$ and $\frac{q-1}{2}$, respectively. (See, for example, [8.1.1]\cite{BH}). 
Since the order of the critical group is determined by the spectrum
we have $\abs{K(P^*(q))} =\abs{K(\Paley(q))}$. 
In \cite[\S2] {CSXPaley} it was shown that the isomorphism type of the Smith group
$S(\Paley(q))$ and that of the $p$-complementary part $K(\Paley(q))_{p'}$ of
the critical group could also be determined from the spectrum and the 
property of being a Cayley graph on an elementary abelian group of order $q$.
The same argument applies to the $p$-complementary part of the Smith group
of all the matrices $A+cI$, where $A$ is the adjacency matrix and $c$ is an integer. 
In particular, for $P^*(q)$, or indeed any cospectral Cayley graph on 
an elementary abelian group of order $q$, these groups are isomorphic to 
the corresponding groups for $\Paley(q)$. Thus, we have the following results.

\begin{theorem} The Smith group of $P^*(q)$ is isomorphic to
$\Z/2r\Z\oplus(\Z/r\Z)^{2r}$, where $r=\frac{q-1}{4}$.
\end{theorem}

\begin{theorem}\label{critgp} Let $K(P^*(q))=K(P^*(q))_p\oplus K(P^*(q))_{p'}$ be the decomposition
of the critical group of $P^*(q)$ into its Sylow $p$-subgroup and $p$-complement. 
Then $K(P^*(q))_{p'}\cong (\Z/r\Z)^{2r}$, where $r=\frac{q-1}{4}$.
The order of  $K(P^*(q))_{p}$ is equal to $q^{\frac{q-3}{2}}$.
\end{theorem}

Later, we shall see that the critical groups of the Paley graphs and Peisert graphs for the same $q$ are generally not isomorphic, although they are isomorphic when $q=p^2$.

\section{The Sylow $p$-subgroup of the critical group}\label{p-part}
We are left with the problem of determining the cyclic decomposition of the 
Sylow $p$-subgroup of $K(P^*(q))$ or, in other words, the
$p$-elementary divisors of $L$.

Let $R_0=\Z[\xi]$, where $\xi$ is a primitive $(q-1)$-st root of unity in an
algebraic closure of $\Q$, and let $\pi$
be a prime ideal of $R_0$ containing $p$. As $p$ is unramified in $R_0$,
in the localization $R=(R_0)_\pi$, the ideal $pR$ is a maximal with $R/pR\cong\Fq$. 
We denote by $v_p(r)$ the $p$-adic valuation of an element $r\in R$. 
Let $R^\Fq$ be the free $R$-module with basis indexed by $\Fq$. For clarity, we
write the basis element corresponding to $x\in \Fq$ as $[x]$.

Let $T:\Fstar\to R^\times$, $T(\beta^j)=\xi^j$, 
be the Teichm\"uller character, which generates the cyclic group $\Hom(\Fstar, R^\times)$.

Then $\Fstar$ acts on $R^\Fq$, which decomposes as the direct sum $R[0]\oplus R^\Fstar$,
and $R^\Fstar$  decomposes further into the direct sum  
of  $\Fstar$-invariant components of rank 1,  affording the characters $T^i$,
$i=0$,\dots,$q-2$. 
The component affording $T^i$ is spanned by
$$
e_i=\sum_{x\in\Fstar}T^i(x^{-1})[x].
$$
Here the subscript $i$ is read modulo $q-1$.
So $R^\Fq$ has  basis $\{e_i \mid i=1,\ldots q-2\} \cup \{e_0, [0]\}$,
where we have separated out the basis for the $\Fstar$-fixed points.

Next consider the action of the subgroup $C_0$. The characters $T^i$,$T^{i+r}$,$T^{i+2r}$, and $T^{i+3r}$ are equal when restricted to  $C_0$ and for $i\notin\{0, r, 2r, 3r\}$
the elements $e_i$, $e_{i+r}$, $e_{i+2r}$ and  $e_{i+3r}$ form a basis
for the $C_0$-isotypic component 
$$
M_i=\{m\in R^\Fq \mid ym=T^i(y)m, \quad\forall y\in C_0\}
$$
of $R^\Fq$ for $1\leq i\leq \frac{q-5}{4}$.
In addition we denote by $M_0$ the isotypic component of the principal character
of $C_0$, namely the submodule of $C_0$-fixed points in $R^\Fq$. As a basis
for $M_0$, we take $\allone=\sum_{x\in \Fq}x=e_0+[0]$, $[0]$, $e_r$, $e_{2r}$ and $e_3r$. 
Thus,
\begin{equation}\label{directsum}
R^\Fq=M_0\oplus\bigoplus_{i=1}^{\frac{q-5}{4}} M_i.
\end{equation}

Since $\mu_LL$ is an $RC_0$-module homomophism, it maps each summand into itself, and
so with respect to the basis formed from the above bases of the $M_i$, the matrix
of $\mu_L$ is block-diagonal with $\frac{q-5}{4}$ $4\times 4$ blocks and a single
$5\times 5$ block. Next we compute these blocks. In these computations,
Jacobi sums will arise, so we recall their definition.
\begin{definition}
Let $\theta$ and  $\psi$ be multiplicative characters of $\Fstar$ taking values in
$R^\times$. By convention, we extend the domain of characters to $\Fq$
by setting the value of the principal character at $0$ 
to be $1$, while nonprincipal characters are assigned the value $0$ there.
The {\it Jacobi sum} is  
\begin{equation}\label{jacobisum}
J(\theta,\psi)=\sum_{x\in\Fq}\theta(x)\psi(1-x). 
\end{equation}
\end{definition}
We refer to \cite[Ch. 2]{Berndt-Evans-Williams} for the elementary formal properties
of Jacobi sums.
At this point we fix some notation for the rest of the paper.
Let $r=\frac{(q-1)}{4}$, $\eta=\xi^r$, $\alpha=\frac{(1-\eta)}{2}$ and $\overline\alpha=\frac{(1+\eta)}{2}$.
Then the characteristic function of $S'$ is
\begin{equation}\label{charfun}
\delta_{S'}=\frac{1}{2}(T^0-\delta_0+\alpha T^r+\overline\alpha T^{-r}),
\end{equation}
where $\delta_0$ takes the value $1$ at $0$ and $0$ on $\Fstar$ and 
(by our convention) the principal character $T^0$ sends all elements of $\Fq$ to $1$.
 We also note for later use that since $q\equiv 1\pmod 8$, we have $T^r(-1)=1$.
\begin{lemma}\label{Lei}
Suppose $i\notin \{0,r,3r\}$. Then 
$$
{\mu_L}(e_i)=\frac{1}{2} (qe_i-\overline\alpha J(T^{-i},T^{-r})e_{i+r}-\alpha J(T^{-i},T^{-3r})e_{i+3r}).
$$
\end{lemma}
\begin{proof}
Since $L=(\frac{q-1}{2})I-A$, we will work with $A$.
By definition of $A$, we have
\begin{equation*}
\begin{aligned}
2{\mu_A}(e_i)&=2\sum_{x\in\Fstar}T^{-i}(x)\sum_{y\in\Fq}\delta_{S'}(y)[x+y]\\
&=\sum_{x\in\Fstar}T^{-i}(x)\sum_{y\in\Fq}(T^0(y)-\delta_0(y)+
\alpha T^r(y)+\overline\alpha T^{-r}(y))[x+y]\\
&=\sum_{x\in\Fstar}T^{-i}(x)\sum_{y\in\Fq}[x+y]
-\sum_{x\in\Fstar}T^{-i}(x)[x]\\
&+\alpha\sum_{x\in\Fstar}T^{-i}(x)\sum_{y\in\Fq}T^r(y)[x+y]
+\overline\alpha\sum_{x\in\Fstar}T^{-i}(x)\sum_{y\in\Fq}T^{-r}(y)[x+y]\\
&=0-e_i+\alpha\sigma+\overline\alpha\sigma',
\end{aligned}
\end{equation*}
where
\begin{equation*}
\sigma=\sum_{x\in\Fstar}T^{-i}(x)\sum_{y\in\Fq}T^r(y)[x+y]\quad\text{and}\quad
\sigma'=\sum_{x\in\Fstar}T^{-i}(x)\sum_{y\in\Fq}T^{-r}(y)[x+y].
\end{equation*}
Then, by substituting $z=x+y$ and changing the order of summation, we have
\begin{equation*}
\begin{aligned}
\sigma&=\sum_{z\in\Fq}\sum_{x\in\Fstar}T^{-i}(x)T^r(z-x)[z]\\
&=\sum_{z\in\Fstar}\sum_{x\in\Fstar}T^{-i}(x)T^r(z-x)[z], 
\end{aligned}
\end{equation*}
as the inner sum vanishes for $z=0$, by the orthogonality of characters.
Then as 
\begin{equation*}
T^{-i}(x)T^r(z-x)=T^{-i}(x/z)T^r(1-x/z)T^{-i+r}(z)
\end{equation*}
we obtain
\begin{equation*}
\sigma=J(T^{-i},T^{-3r})e_{i+3r}.
\end{equation*}
Similarly, $\sigma'=J(T^{-i},T^{-r})e_{i+r}$. 
\end{proof}

\begin{lemma}\label{Ler}
\begin{itemize}
\item[(i)] $\mu_L(\allone)=0$.
\item[(ii)] $\mu_L([0])=\frac{1}{2}(-\allone+ q[0]-\overline\alpha e_r-\alpha e_{3r}).$
\item[(iii)] $\mu_L(e_r)=\frac{1}{2}(\alpha\allone-q\alpha[0]+q e_r-\overline\alpha J(T^{-r},T^{-r}) e_{2r})$.
\item[(iv)] $\mu_L(e_{2r})=\frac{1}{2}(-\alpha J(T^{-2r},T^{-3r})e_r +q e_{2r}
-\overline\alpha J(T^{-2r},T^{-r})e_{3r})$.
\item[(v)] $\mu_L(e_{3r})=\frac{1}{2}(\overline\alpha\allone-q\overline\alpha[0]-\alpha J(T^{-3r},T^{-3r})e_{2r}+qe_{3r}).$
\end{itemize}
\end{lemma}
\begin{proof} As $L=(\frac{q-1}{2})I-A$, it is enough to compute $2\mu_A$ on the basis elements.Part (i) is obvious. For (ii) we have
\begin{equation*}
\begin{aligned}
2\mu_A([0])&=2\sum_{y\in\Fq}\chi_{S'}(y)[y]\\
&=\allone -[0]+\alpha\sum_{y\in\Fq}T^{-3r}(y)[y]+\overline\alpha\sum_{y\in\Fq}T^{-r}(y)[y]\\
&=\allone-[0]+\alpha e_{3r}+\overline\alpha e_{r}.
\end{aligned}
\end{equation*}
Part (iv) is the case $i=2r$ of  Lemma~\ref{Lei}. It remains to prove
(iii) and (v). It suffices to prove (iii) since the two cases are related by an automorphism of $\Fq$.
By definition of $A$, we have
\begin{equation*}
\begin{aligned}
2\mu_A(e_{r})&=2\sum_{x\in\Fstar}T^{-r}(x)\sum_{y\in\Fq}\delta_{S'}(y)[x+y]\\
 &=\sum_{x\in\Fstar}T^{-r}(x)\sum_{y\in\Fq}(T^0(y)-\delta_0(y)+
\alpha T^r(y)+\overline\alpha T^{-r}(y))[x+y]\\
&=\sum_{x\in\Fstar}T^{-r}(x)\sum_{y\in\Fq}[x+y]
-\sum_{x\in\Fstar}T^{-r}(x)[x]\\
&+\alpha\sum_{x\in\Fstar}T^{-r}(x)\sum_{y\in\Fq}T^r(y)[x+y]
+\overline\alpha\sum_{x\in\Fstar}T^{-r}(x)\sum_{y\in\Fq}T^{-r}(y)[x+y]\\
&=0-e_{r}+\alpha\sigma+\overline\alpha\sigma',
\end{aligned}
\end{equation*}
where
\begin{equation*}
\sigma=\sum_{x\in\Fstar}T^{-r}(x)\sum_{y\in\Fq}T^r(y)[x+y]\quad\text{and}\quad
\sigma'=\sum_{x\in\Fstar}T^{-r}(x)\sum_{y\in\Fq}T^{-r}(y)[x+y].
\end{equation*}
Then, by substituting $z=x+y$ and changing the order of summation, we have
\begin{equation*}
\begin{aligned}
\sigma&=\sum_{z\in\Fq}\sum_{x\in\Fstar}T^{-r}(x)T^r(z-x)[z]\\
&=(q-1)[0]+\sum_{z\in\Fstar}\sum_{x\in\Fstar}T^{-r}(x)T^r(z-x)[z],
\end{aligned}
\end{equation*}
as the inner sum when $z=0$ is $(q-1)[0]$.

For $z\neq 0$ we have
\begin{equation*}
T^{-r}(x)T^r(z-x)=T^{-r}(x/z)T^r(1-x/z)T^{-4r}(z)
\end{equation*}
and so
\begin{equation*}
\sum_{x\in\Fstar} T^{-r}(x/z)T^r(1-x/z)T^{-4r}(z)=J(T^{-r},T^{r})[z]
=-T^r(-1)[z]=-[z]
\end{equation*}
So
\begin{equation*}
\sum_{z\in\Fstar}\sum_{x\in\Fstar}T^{-r}(x)T^r(z-x)[z]=-(\allone-[0]).
\end{equation*}
Thus, $\sigma=q[0]-\allone$. We now turn to $\sigma'$.
By substituting $z=x+y$ and changing the order of summation, we have
\begin{equation*}
\begin{aligned}
\sigma'&=\sum_{z\in\Fq}\sum_{x\in\Fstar}T^{-r}(x)T^{-r}(z-x)[z]\\
&=\sum_{z\in\Fstar}\sum_{x\in\Fstar}T^{-r}(x)T^{-r}(z-x)[z],
\end{aligned}
\end{equation*}
as the inner sum when $z=0$ is $0$. 
For $z\neq 0$ we have
\begin{equation*}
T^{-r}(x)T^{-r}(z-x)=T^{-r}(x/z)T^{-r}(1-x/z)T^{-2r}(z),
\end{equation*}
and so
\begin{equation*}
\sigma'=J(T^{-r},T^{-r})e_{2r}.
\end{equation*}
\end{proof}

An integer $j$ which is not divisible by $q-1$ has, when reduced modulo 
$q-1$, a unique $p$-digit expresssion $j=a_0+a_1p+a_2p^2+\cdots+a_{2t-1}p^{2t-1}$, where $0\leq a_i\leq p-1$. We shall write this $p$-digit expression a $2t$-tuple $(a_0,a_1,\ldots,a_{2t-1})$. Let $s(j)$ denote the sum $\sum_ia_i$ of the $p$-digits of $j$ modulo $q-1$. In this notation, the tuple for $r=\frac{q-1}{4}$ has $a_i=\frac{3p-1}{4}$
for even $i$ and $a_i=\frac{p-3}{4}$ for odd $i$, while
the tuple for $3r$ have the same entries but in the positions of opposite
parity. We have $s(r)=s(3r)=t(p-1)$. The $p$-digits of $2r$ are all $\frac{p-1}{2}$,
so $s(2r)=t(p-1)$ also.
 
By Stickelberger's Theorem \cite{stick} (see \cite[p.~636]{dwork} for further reference)
and the relation between Gauss sums and Jacobi sums,
we know that when $i$, $j$ and $i+j$ are not divisible 
by $q-1$ the $p$-adic valuation of $J(T^{-i},T^{-j})$  is equal to 
\begin{equation*}
c(i,j):=\frac{1}{p-1}(s(i)+s(j)-s(i+j)),
\end{equation*}
This valuation can be viewed as  the number of carries, when adding the $p$-expansions of $i$ and $j$, modulo $q-1$. 

The following equations are immediate.
\begin{lemma} \label{carries}
Suppose $1\leq i\leq q-2$ and $i\neq r$, $2r$, $3r$. Then
\begin{itemize}
\item[(i)] $c(i,r)+c(q-1-i,r)=2t$.
\item[(ii)]  $c(i,r)+c(i+r,3r)+c(i+2r,r)+c(i+3r,3r)=4t$.
\item[(iii)]$c(i,r)+c(i+2r,r)=c(i,3r)+c(i+2r,3r)$.
\end{itemize}
\end{lemma}

\begin{theorem}\label{main}
\begin{enumerate}
\item The $p$-elementary divisors of $(\mu_L)_{\mid M_0}$ are $0$, $1$, $1$, $p^t$, $p^t$.
\item For $1\leq i\leq \frac{q-5}{4}$, 
consider the two lists $\{c(i,r), c(i+r,3r), c(i+2r,r), c(i+3r,3r)\}$
and
$\{c(i,3r), c(i+r,r), c(i+2r,3r), c(i+3r,r)\}$
and let $C_i$ be the list that contains the smallest element.
Then the four $p$-elementary divisors of $(\mu_L)_{\vert M_i}$ are $p^c$ for $c$ in $C_i$.
\end{enumerate}
\end{theorem}
\begin{proof}
If $X$ is a matrix with entries in $R$ or a homomorphism of finitely generated, free $R$-modules, we   let $m_j(X)$ denote the multiplicity 
of $p^j$ as a $p$-elementary divisor and  let $\kappa(X)$ denote the 
product of the nonzero $p$-elementary divisors. 
Thus $v_p(\kappa(X))=\sum_jjm_j(X)$, and in the case of our Laplacian matrix,
$\kappa(L)=\kappa(\mu_L)$ is the order of the $p$-Sylow subgroup of the critical group.
We first note that for any given power $p^s$, if two matrices $X$ and $X'$ over $R$ are equal modulo $p^s$ then $m_j(X)=m_j(X')$ for every $j<s$. 
Also, we have 
\begin{equation}\label{lowerbound}
v_p(\kappa(X))\geq \sum_{j=0}^{s-1} jm_j(X)+s(\rank(X)-\sum_{j=0}^{s-1}m_j(X)).
\end{equation}
Our proof will make use of these general facts in the following way.
We shall obtain a lower bound for $v_p(\kappa(L))$
by looking at the matrix of $\mu_L$ modulo $q$. Then we shall see that as this lower bound
coincides with the actual value of $v_p(\kappa(L))$ known from the Matrix-Tree Theorem,
we must actually have equality in several inequalities used to deduce the lower bound.
These inferences will enable us to complete the proof.

The matrix of $2{\mu_L}_{\vert M_i}$ is
\begin{equation}\label{Mi}
\begin{bmatrix}
q&-\alpha J(T^{-i-r},T^{-3r})&0&-\overline\alpha J(T^{-i-3r},T^{-r})\\
-\overline\alpha J(T^{-i},T^{-r})&q&-\alpha J(T^{-i-2r},T^{-3r})&0\\
0&-\overline\alpha J(T^{-i-r},T^{-r})&q&-\alpha J(T^{-i-3r},T^{-3r})\\
-\alpha J(T^{-i},T^{-3r})&0&-\overline\alpha J(T^{-i-2r},T^{-r})&q\\
\end{bmatrix}
\end{equation}

If we work modulo $q$, this matrix is $R$-equivalent to 
\begin{equation}\label{blockform}
B=\begin{bmatrix}
u_{11}J(T^{-i},T^{-r})&u_{12}J(T^{-i-2r},T^{-3r})&0&0\\
u_{21}J(T^{-i},T^{-3r})&u_{22}J(T^{-i-2r},T^{-r})&0&0\\
0&0&v_{11}J(T^{-i-3r},T^{-3r})&v_{12}J(T^{-i-3r},T^{-r})\\
0&0&v_{21}J(T^{-i-r},T^{-r})&v_{22}J(T^{-i-r},T^{-3r}),
\end{bmatrix}
\end{equation}
where the $u_{mn}$ and $v_{mn}$ are units of $R$.

To apply Lemma~\ref{carries} it is helpful to consider the matrix 
\begin{equation}\label{blockvals}
C=\begin{bmatrix}
c(i,r)&c(i+2r,3r)&\cdot&\cdot\\
c(i,3r)&c(i+2r,r)&\cdot&\cdot\\
\cdot&\cdot&c(i+3r,3r)&c(i+3r,r)\\
\cdot&\cdot&c(i+r,r)&c(i+r,3r)
\end{bmatrix}
\end{equation}
of the valuations of the nonzero entries of $B$. As these entries are integers
in the range $[0,2t]$, we can apply (\ref{lowerbound}) with $X={\mu_L}_{\vert M_i}$,
$X'=B$ and $s=2t$ to obtain $v_p(\kappa({\mu_L}_{\vert M_i})))\geq v_p(\kappa(B))$.
By  Lemma~\ref{carries}(iii), the diagonal sum of each $2\times 2$ block
is equal to the anti-diagonal sum. It follows that 
\begin{equation}\label{valLMi}
v_p(\kappa(B))\geq c(i,r)+c(i+r,3r)+c(i+2r,r)+c(i+3r,3r)=4t.
\end{equation}
where the last equality is by Lemma~\ref{carries}(ii).

Suppose that equality holds. Then the determinants of
each $2\times 2$ block of $B$  must have $p$-adic valuations
exactly equal to the sums of the corresponding diagonals (or anti-diagonals)
in $C$. Once  the $p$-adic valuation of determinant
of a $2\times 2$ matrix is known, then its  $p$-elementary divisors
will be determined by the smallest among the $p$-adic valuations of its  entries.
This shows  that the $p$-elementary divisors of $B$ 
will be determined by the minimum valuation of an entry in each of the two blocks.  
However, we can say more, since it also  follows from the definitions
of $c(i)$ and the fact that $s(r)=s(3r)=t(p-1)$
that the each entry of the upper block of $C$
can be obtained from corresponding  entry of the lower block by adding
$\frac{1}{p-1}(s(i)-s(i+r)+s(i+2r)-s(i+3r))$.  Thus, the lowest $p$-adic valuations
of entries occur in the same position in the two blocks. It follows that,
if the lowest $p$-adic valuation occurs for a diagonal entry, then
the $p$-elementary divisors of $B$ are
\begin{equation*}
p^{c(i,r)},\quad p^{c(i+2r,r)},\quad p^{c(i+r,3r)},\quad p^{c(i+3r,3r)},
\end{equation*}
while if the lowest $p$-adic valuation occurs for an  anti-diagonal entry, then
the $p$-elementary divisors of $B$ are
\begin{equation*}
p^{c(i+r,3r)},\quad p^{c(i+2r,3r)},\quad p^{c(i+r,r)},\quad p^{c(i+3r,r)}.
\end{equation*}
We conclude that, under the assumption of equality in (\ref{valLMi}),
the $p$-elementary divisors of $B$ are determined by the smallest
$p$-adic valuation of an entry. 

The matrix $2{\mu_L}_{\vert M_0}$ is
\begin{equation}\label{Mzero}
\begin{bmatrix}
0&-1&\alpha&0&\overline\alpha\\
0&q&-q\alpha&0&-q\overline\alpha\\
0&-\overline\alpha&q&-\alpha J(T^{-2r},T^{-3r})&0\\
0&0&-\overline\alpha J(T^{-r},T^{-r})&q&-\alpha J(T^{-3r},T^{-3r})\\
0&-\alpha&0&-\overline\alpha J(T^{-2r},T^{-r})&q
\end{bmatrix}
\end{equation}
Modulo $q$, it is $R$-equivalent to 
\begin{equation}\label{Mzero1}
\begin{bmatrix}
0&0&0&0&0\\
0&-1&\alpha&0&\overline\alpha\\
0&-\overline\alpha&0&-\alpha J(T^{-2r},T^{-3r})&0\\
0&0&-\overline\alpha J(T^{-r},T^{-r})&0&-\alpha J(T^{-3r},T^{-3r})\\
0&-\alpha&0&-\overline\alpha J(T^{-2r},T^{-r})&0
\end{bmatrix}.
\end{equation}
The lower $4\times 4$ submatrix
\begin{equation}\label{Mzero2}
\begin{bmatrix}
-1&\alpha&0&\overline\alpha\\
-\overline\alpha&0&-\alpha J(T^{-2r},T^{-3r})&0\\
0&-\overline\alpha J(T^{-r},T^{-r})&0&-\alpha J(T^{-3r},T^{-3r})\\
-\alpha&0&-\overline\alpha J(T^{-2r},T^{-r})&0
\end{bmatrix}
\end{equation}
can be reduced by elementary row and column operations to
\begin{equation}\label{Mzero3}
\begin{bmatrix}
1&0&0&0\\
0&\overline\alpha\alpha&\alpha J(T^{-2r},T^{-3r})&\overline\alpha^2\\
0&\overline\alpha J(T^{-r},T^{-r})&0&\alpha J(T^{-3r},T^{-3r})\\
0&\alpha^2&\overline\alpha J(T^{-2r},T^{-r})&\overline\alpha\alpha
\end{bmatrix}
\end{equation}
and the lower $3\times 3$ block cam be further reduced to
\begin{equation}\label{Mzero4}
\begin{bmatrix}
\alpha\overline\alpha&\alpha J(T^{-2r},T^{-3r})&\overline\alpha^2\\
0&-\alpha J(T^{-r},T^{-r})J(T^{-2r},T^{-3r})&\alpha^2J(T^{-3r},T^{-3r})-\overline\alpha^2J(T^{-r},T^{-r})\\
0&\overline\alpha^2 J(T^{-2r},T^{-r})-\alpha^2 J(T^{-2r},T^{-3r})&0
\end{bmatrix}.
\end{equation}
Since $c(r,r)=c(2r,3r)=c(2r,r)=c(3r,3r)=t$, we see from this last
matrix form that $v_p(\kappa({\mu_L}_{\vert M_0}))\geq 2t$, with equality if
and only the $p$-elementary divisors are $0$, $1$, $1$, $p^t$ and $p^t$.

Combining our bounds for $v_p(\kappa({\mu_L}_{\vert M_0}))$ and 
$v_p(\kappa({\mu_L}_{\vert M_i}))$ we see that 
\begin{equation}
v_p(\kappa(\mu_L))=v_p(\kappa({\mu_L}_{\vert M_0}))+\sum_{i=1}^{\frac{q-5}{4}}\geq
2t+{\frac{q-5}{4}}4t=(q-3)t= v_p(\kappa(\mu_L)),
\end{equation}
where the last equality is from Theorem~\ref{critgp}. Thus, all of
our inequalities are equalities and the theorem is proved.
\end{proof}

\begin{corollary}\label{palindromic} Let $m(i)$ denote the multiplicity
of $p^i$ as a $p$-elementary divisor of $L$. Then for $1\leq i\leq 2t-1$ 
we have $m(i)=m(2t-i)$, and  $m(0)=m(2t)+2$.
\end{corollary}
\begin{proof} The corollary follows from Theorem~\ref{main} and
Lemma~\ref{carries}(i).
\end{proof}

We can also obtain the $p$-rank, which was first computed in  \cite[Theorem 3.4]{WQWX}.
\begin{corollary} $\rank_p L=2(3^t-1)(\frac{p+1}{4})^{2t}$
\end{corollary} 
\begin{proof} 
We know that the $p$-rank of ${\mu_L}_{\vert M_0}$ is $2$, so we need to count
the occurrences of $1$ as a $p$-elementary divisor in the ${\mu_L}_{\vert M_i}$
for $1\leq i\leq \frac{q-5}{4}$. 

We note that if we  swap the rows and columns of the lower block
of the matrix  $C$ in  (\ref{blockvals}) (which corresponds to
swapping the same rows and colums in (\ref{blockform}) )
we obtain a block sum of two matrices, both of the form
\begin{equation}\label{exponentmatrix}
\begin{bmatrix}
c(j,r)&c(j+2r,3r)\\
c(j,3r)&c(j+2r,r)
\end{bmatrix}
\end{equation}
for suitable $j$, and that as $j$ runs from $1$ to $\frac{q-5}{4}$,  
the entries in the blocks form the multiset 
$$
\{c(\ell,r), c(\ell,3r) \mid\, 1\leq \ell\leq q-1, \ell\neq 0,r,2r,3r \}
$$
By examining just the $0$-th $p$-digit,
it is easy to see that if $c(j,r)=0$, then $c(j+2r,r)>0$. Likewise,
if $c(j,3r)=0$, then $c(j+2r,3r)>0$.  
This means that (\ref{exponentmatrix})
has at most one zero on the diagonal and at most one zero on the anti-diagonal.
In view of Lemma\ref{carries}(iii), there can be at most
one $p$-elementary divisor equal to 1 in the corresponding block (\ref{blockform}),
and this will occur if and only if there is at least one zero
entry in (\ref{exponentmatrix}).
With these observations in hand, it is now a simple matter to count
the number of blocks with a nonzero entry by
counting the sets $\{i \mid\, c(i,r)=0\}$, $\{i \mid\, c(i,3r)=0\}$,
$\{i \mid\, c(i,r)=0 \ \text{and}\  c(i,3r)=0\}$,  and $\{i \mid\, c(i,r)=0\ \text{and}\  c(i+2r,3r)=0\}$.
The first set consists of those $i$ whose even index $p$-digits are $\leq \frac{3p-1}{4}$
and whose odd index $p$-digits are  $\leq \frac{p-3}{4}$, so this set has size
$(\frac{3(p+1)}{4})^t(\frac{(p+1)}{4})^t$. Similarly, the second set has the same size,
while the last two sets have size $(\frac{(p+1)}{4})^{2t}$. The result follows. 
\end{proof}

The following examples give an idea of the size and structure of the critical groups. 

Our first example provides an alternative proof to
the one in \cite{Peisert} that $P^*(9^2)$ and $\Paley(9^2)$ are not isomorphic.

\begin{example} Let $q=9^2$. Then from \cite{CSXPaley}, we have
\begin{equation*}
K(\Paley(9^2))\cong (\Z/20\Z)^{40}\oplus[(\Z/3Z)^{16}\oplus(\Z/9\Z)^{18}\oplus(\Z/27\Z)^{16}\oplus(\Z/81\Z)^{14}],
\end{equation*}
while the results of the present paper show
\begin{equation*} 
K(P^*(9^2))\cong (\Z/20\Z)^{40}\oplus[(\Z/3Z)^{20}\oplus(\Z/9\Z)^{10}\oplus(\Z/27\Z)^{20}\oplus(\Z/81\Z)^{14}].
\end{equation*}
\end{example}

Using Theorem~\ref{main} we can compute the critical groups
of larger examples than would be possible by working directly with the Laplacian matrix.

\begin{example} The critical group $K(P^*(3^{12}))$ is isomorphic to
\begin{equation*}
\begin{aligned} 
(\Z/132860\Z)^{265720}&\oplus[(\Z/{3}\Z)^{11376}\oplus (\Z/{3^2}\Z)^{33408}
\oplus (\Z/{3^3}\Z)^{54176}\oplus(\Z/{3^4}\Z)^{66852}\\
&\oplus(\Z/{3^5}\Z)^{66420}\oplus(\Z/{3^6}\Z)^{64066}
\oplus (\Z/{3^7}\Z)^{66420}\oplus(\Z/{3^8}\Z)^{66852}\\
&\oplus(\Z/{3^9}\Z)^{54176}\oplus(\Z/{3^{10}}\Z)^{33408}\oplus(\Z/{3^{11}}\Z)^{11376}\oplus(\Z/{3^{12}}\Z)^{1454}].
\end{aligned}
\end{equation*}
\end{example}

\section{ Paley and Peisert graphs with over fields of order $p^2$, with
$p\equiv 3\pmod 4$.}
Assume that $q=p^{2t}$, $p\equiv 3\pmod4$. In this section we shall use
$A(q)$ to denote the adjacency matrix of $\Paley(q)$ and $A^*(q)$ 
to denote the adjacency matrix of $P^*(q)$, both with respect to
some arbitrary but fixed ordering on their common vertex set $\Fq$.
The graphs $\Paley(q)$ 
and $P^*(q)$ are cospectral, but in the special
case when $q=p^2$, we shall show that they are even more closely related.

Let $D$ be an integral domain and let $D_n$ denote the ring of $n\times n$ matrices with entries in $D$. We shall say that two matrices $A$ and $B$ in $D_n$ are {\it similar over $D$}
if, and only if, there is an invertible element $C$ of $D_n$ such that $CAC^{-1}=B$. 
If $P$ is a prime ideal of $D$, we denote be $D_P$ the localization of $D$ at $P$.

\begin{theorem}\label{similar} Assume $q=p^{2}$ with $p\equiv 3\pmod4$.
\begin{enumerate}
\item[(i)] $A(q)$ and $A^*(q)$ are similar over the ring  of algebraic integers
in some number field.
\item[(ii)] $A(q)$ and $A^*(q)$ are similar over the $\ell$-local integers 
$\Z_{(\ell)}$ for every prime $\ell\in\Z$.
\end{enumerate}
\end{theorem}

Before giving the proof of Theorem~\ref{similar} we discuss its implications.
Since the Smith normal form of a matrix is determined
locally, that is, one prime at a time, any two matrices that satisfy
the equivalent conditions of Theorem~\ref{similar} have the same
Smith normal form. 

By Theorem~\ref{similar}, it is immediate that for any $a$, $b\in \Z$ 
the matrices $aA(q)+bI$ and $aA^*(q)+bI$ are cospectral and have the same
Smith normal form. Since $\Paley(q)$ and $P^*(q)$ are strongly regular
graphs with the same parameters $(k,\lambda,\mu)=(\frac{q-1}{2},\frac{q-5}{4},\frac{q-1}{4})$,
 the equation 
\begin{equation}
A^2+(\mu-\lambda)A+(\mu-k)I=\mu J,
\end{equation}
satisfied by both $A(q)$ and $A^*(q)$, implies that any matrix $C$
with $CA(q)C^{-1}=A^*(q)$ must commute with $J$ and therefore transforms 
the generalized adjacency matrix $aA(q)+bI+cJ$ to  $aA^*(q)+bI+cJ$ for any $a$, $b$ and $c$. 
We arrive at the following corollary.

\begin{corollary} Let $q=p^2$, $p\equiv 3\pmod4$.
For any integers $a$, $b$ and $c$, the generalized adjacency matrices
$aA(q)+bI+cJ$ and $aA^*(q)+bI+cJ$ are cospectral and have the same
Smith normal forms.\qed
\end{corollary}

We now turn to the proof of Theorem~\ref{similar}.
We shall make use of the following ``local-global'' theorem of Guralnick \cite[Theorem 7]{Guralnick}, based on results of Reiner-Zassenhaus \cite{RZ}, Taussky \cite{Taussky} and Dade \cite{Dade}. 

\begin{theorem}\label{local-global}Let $D$ be the ring of integers is a finite extension of $\Q$.
Suppose $A$, $B\in D_n$. Then the following are equivalent.
\begin{enumerate}
\item[(i)] $A$ and $B$ are similar over $D_P$ for each prime ideal $P$ of $D$.
\item[(ii)]$A$ and $B$ are similar over some finite integral extension of $D$.
\end{enumerate}\qed
\end{theorem}

\begin{lemma}\label{ellsimilar} Assume $q=p^{2t}$ with $p\equiv 3\pmod4$.
Let $\ell\neq p$ be a prime and let $\Lambda$ be a prime ideal lying over $\ell$
in the cyclotomic ring $\Z[\zeta]$, where $\zeta$ is a primitive $p$-th root of unity
in an algebraic closure of $\Q$.
Then $A(q)$ and $A^*(q)$ are similar over the localization $\Z[\zeta]_\Lambda$.
\end{lemma}
\begin{proof}
Let $X$ be the character table of $(\Fq,+)$, considered as a matrix with entries
in $\Z[\zeta]$. By the orthogonality relations $X$ is invertible over the ring
$\Z[\zeta][\frac{1}{p}]$. It has long been known (cf. \cite{McWM}) that
the adjacency matrix of an abelian Cayley graph can be transformed to
diagonal form using the character table, so
\begin{equation}
XA(q)X^{-1}=E,\qquad \text{and}\qquad  XA^*(q)X^{-1}=E^*,
\end{equation}
where $E$ and $E^*$ are the diagonal matrices of eigenvalues in some order.
Since $A(q)$ and $A^*(q)$ are cospectral, there is a permutation matrix $P$
such that $PEP^{-1}=E^*$.  Therefore, we have
\begin{equation}
(X^{-1}PX)A(q)(X^{-1}PX)^{-1}=A^*(q).
\end{equation}
We may regard this as an equation in the ring of matrices over $\Z[\zeta][\frac{1}{p}]$,
and since for every $\ell\neq p$ and any prime ideal $\Lambda$ of $\Z[\zeta]$
containing $\ell$, we have $\Z[\zeta][\frac{1}{p}]\subseteq \Z[\zeta]_\Lambda$,
the lemma is proved.
\end{proof}

In order to complete the proof of Theorem~\ref{similar} it suffices to show
that,  that $A(p^2)$ and $A^*(p^2)$ are similar over the ring  $R$ of \S\ref{p-part}. 
For then we may apply Theorem~\ref{local-global} first with $D$ being the
ring of integers in $\Q(\zeta,\xi)$ to deduce (i), and a second time  with 
with $D=\Z$ to deduce (ii).
We assume $p$ to be fixed from now on.
As similarity of $A(p^2)$ and $A^*(p^2)$ is equivalent to similarity
of $K=2A(p^2)+I$  and $K^*=2A^*(p^2)+I$,  we shall consider
the latter matrices, as they have a more convenient form.

By \cite[Lemma 3.1]{CSXPaley}, the matrix of $\mu_K$ on $M_i$ with respect to
the ordered basis $e_i$, $e_{i+2r}$, $e_{i+r}$, $e_{i+3r}$ 
is 

\begin{equation}\label{Ki}
K_i=\begin{bmatrix}
0&J(i+2r,2r)&0&0\\
J(i,2r)&0&0&0\\
0&0&0&J(i+3r,2r)\\
0&0&J(i+r,2r)&0\\
\end{bmatrix}
\end{equation}

The matrix of $\mu_{K^*}$ on $M_i$ with respect to
the ordered basis 
$e_i$, $e_{i+2r}$, $e_{i+r}$, $e_{i+3r}$ is
is
\begin{equation}\label{Kistar}
K^*_i=\begin{bmatrix}
0&0&\alpha J(i+r,3r)&\overline\alpha J(i+3r,r)\\
0&0&\overline\alpha J(i+r,r)&\alpha J(i+3r,3r)\\
\overline\alpha J(i,r)&\alpha J(i+2r,3r)&0&0\\
\alpha J(i,3r)&\overline\alpha J(i+2r,r)&0&0\\
\end{bmatrix}
\end{equation}

The matrix of $\mu_{K}$ on $M_0$ with respect to
the ordered basis $\allone$, $[0]$, $e_{2r}$, $e_r$, $e_{3r}$
is
\begin{equation}
K_0=\begin{bmatrix}
q&1&-1&0&0\\
0&0&q&0&0\\
0&1&0&0&0\\
0&0&0&0&J(3r,2r)\\
0&0&0&J(r,2r)&0
\end{bmatrix}
\end{equation}
The matrix $\mu_{K^*}$ on $M_0$ with respect to
the ordered basis $\allone$, $[0]$, $e_r$, $e_{2r}$,  $e_{3r}$
is
\begin{equation}
K^*_0=\begin{bmatrix}
q&1&-\alpha&0&-\overline\alpha\\
0&0&q\alpha&0&q\overline\alpha\\
0&\overline\alpha&0&\alpha J(2r,3r)&0\\
0&0&\overline\alpha J(r,r)&0&\alpha J(3r,3r)\\
0&\alpha&0&\overline\alpha J(2r,r)&0
\end{bmatrix}
\end{equation}

Our aim is to show that $K_i$ and $K_i^*$ are similar over $R$. 
We first dispose of the similarity of $K_0$ and $K_0^*$.

We shall need some results on Gauss and Jacobi sums over the field of
$p^2$ elements, which follow immediately from \cite[Theorem 2.12]{Berndt-Evans}
and the well known formula expressing a Jacobi sum of two characters as the product
of their  Gauss sums divided by the Gauss sum of their product character.
\begin{lemma} \label{berndt}  
\item[(i)] $J(r,r)=J(3r,3r)=J(r,2r)=J(3r,2r)=p$.
\item[(ii)]For $1\leq i\leq q-2$ and $i\notin\{r,2r,3r\}$ we have
$J(i,r)J(i+r,r)=J(i,3r)J(i+3r,3r)$.
\qed
\end{lemma}

Let $v_1=\allone$, $v_2=[0]$,
$v_3=\overline\alpha e_r+\alpha e_{3r}$, $v_4=e_{2r}$,  
$v_5=\alpha e_r+\overline\alpha e_{3r}$.
Using the relations $\alpha^2=-\frac\eta2$, $\overline\alpha^2=\frac\eta2$ and
$\alpha\overline\alpha=\frac12$, and Lemma~\ref{berndt} it is easy to check
that indeed the $v_i$ form a basis of $M_0$ and that the matrix of $\mu_{K^*}$ 
on $M_0$ in this new basis is the matrix $K_0$. We have  thus established the similarity of $K_0$ and $K_0^*$.

\begin{lemma}\label{eigenvalues}
\begin{enumerate}
\item[(i)] For $1\leq i\leq \frac{q-1}{4}$ the eigenvalues of each $2\times 2$ block of $K_i$ are $p$ and $-p$.
\item[(ii)] For $1\leq i\leq q-2$ and $i\notin\{r,2r,3r\}$ we 
have $J(i,2r)J(i+2r,2r)=p^2$.
\item[(iii)] The eigenvalues of $K_i^*$ are $p$ and $-p$, each with multiplicity 2.
\end{enumerate}
\end{lemma}
\begin{proof}
The eigenvalues of $K$ and $K^*$ on $R^{\Fq}$ are $p^2$, with multiplicity $1$
and eigenvector $\allone$, and $p$ and $-p$, with equal multiplicity.
It follows that on any invariant subspace not containing $\allone$ on which
$K$ (respectively $K^*$) has trace $0$, the eigenvalues are $p$ and $-p$
with equal multiplicity, so  (i) and (iii) hold. Then (ii) follows
from (i) and (\ref{Ki}).
\end{proof}

\begin{lemma} Let $1\leq i\leq \frac{q-1}{4}$. Then $K_i$
is similar to the matrix in the follwing list which has the same $p$-rank
as $K_i$.
\begin{equation}\label{canon}
\begin{bmatrix}
0&p&0&0\\
p&0&0&0\\
0&0&0&p\\
0&0&p&0
\end{bmatrix},
\begin{bmatrix}
0&p^2&0&0\\
1&0&0&0\\
0&0&0&p\\
0&0&p&0
\end{bmatrix},
\begin{bmatrix}
0&p^2&0&0\\
1&0&0&0\\
0&0&0&p^2\\
0&0&1&0
\end{bmatrix}.
\end{equation}
\end{lemma}
\begin{proof}
We choose a new basis $v_1=e_i$, $v_2=p^{-c(i,2r)}J(i,2r)e_{i+2r}$,  $v_3=e_{i+r}$, $v_4=p^{-c(i+r,2r)}J(i+r,2r)e_{i+3r}$. Then, by Lemma~\ref{eigenvalues}(ii),
the matrix of $\mu_K$ on $M_i$ is the matrix in  (\ref{canon}) that has the same $p$-rank
as $K_i$.
\end{proof}

For $1\leq i\leq \frac{q-1}{4}$, let $J(i)=\{i,i+r,i+2r,i+3r\}$
Then for $j\in J(i)$, the vectors $e_j$, $e_{j+r}$, 
$e_{j+2r}$, $e_{j+r}$ are just the vectors $e_i$, $e_{i+r}$,
$e_{i+2r}$, $e_{i+2r}$ is a different order, so the matrix, which we shall call
$K^*_j$ of $\mu_{K^*}$ on $M_i$  with respect to the first ordered basis
is similar to $K^*_i$ (by a permutation matrix). 
This gives us the flexibility to talk about $K^*_j$ for any
$j$ with $1\leq j\leq q-2$, $j\notin\{r,2r,3r\}$. We will show 
for each $i$ with $1\leq i\leq \frac{q-1}{4}$, that $K_j^*$ is
similar over $R$ to $K_i$ for some $j\in J(i)$.

We consider the matrix of $p$-adic valuations of the entries in $K_i^*$:

\begin{equation}\label{vistar}
\begin{bmatrix}
0&0&c(i+r,3r)&c(i+3r,r)\\
0&0&c(i+r,r)&c(i+3r,3r)\\
c(i,r)&c(i+2r,3r)&0&0\\
c(i,3r)&c(i+2r,r)&0&0\\
\end{bmatrix}.
\end{equation}

From the definition of $c(i,j)$ and the fact that $s(r)=s(3r)=(p-1)$ can be written in 
the form
\begin{equation}\label{abcdD}
\begin{bmatrix}
0&0&d+D&b+D\\
0&0&c+D&a+D\\
a&b&0&0\\
c&d&0&0\\
\end{bmatrix},
\end{equation}
where $D=\frac{1}{p-1}(s(i)-s(i+r)+s(i+2r)-s(i+3r))$ and 
the entries $a$, $b$, $c$, $d$, $a+D$, $b+D$, $c+D$, $d+D$ lie in the set $\{0,1,2\}$ .

\begin{lemma}\label{abcd}
\begin{enumerate}
\item[(i)] $a+d=b+c$.
\item[(ii)] $a+d+b+c+2D=4$, so $a+d+D=2$.
\item[(iii)] $D\in\{-1,0,1\}$. Hence $a+d>0$.
\end{enumerate}
\end{lemma}
\begin{proof} Parts (i) and (ii) are special cases of parts (ii) and (iii)
of Lemma~\ref{carries}. By (ii) we know $\abs{D}\leq 2$. 
If $D=2$, then by (ii), we must have $a=b=c=d=0$.
If $D=-2$, then by (ii), we must have $a=b=c=d=2$ and the upper right matrix
is zero. So, by replacing $i$ by $i+r$ if necessary, we can assume
that $D=2$ and that $a=b=c=d=0$, in order to reach a contradiction.
Now $a=c(i,r)$ and $c=c(i,3r)$. Let $i=(i_0,i_1)$ and recall that
$r=(\frac{3p-1}{4},\frac{p-3}{4})$ and $r=(\frac{p-3}{4},\frac{3p-1}{4})$.
In order for $a=c=0$, we must have $0\leq i_0$,$i_1\leq \frac{p-3}{4}$. But
then if $i+2r=(j_0,j_1)$, we have $\frac{p-1}{2}\leq j_0$, $j_1$, which forces
$c(i+2r,r)>0$, contrary to the assumption that $d=0$. The final
assertion is clear. 
\end{proof}

\begin{lemma} $K_i^*$ and $K_i$ have the same $p$-rank, for $1\leq i\leq \frac{q-1}{4}$.
\end{lemma}
\begin{proof}
We have already seen that the $\rank_p (K_i)\in\{0, 1, 2\}$.
To see that $\rank_p (K_i^*)\in\{0, 1, 2\}$, we simply note that 
each of the  anti-diagonal $2\times 2$ blocks in $K_i^*$ must be
singular modulo $p$, by Lemma~\ref{abcd}(iii).
We will prove that $\rank_p(K_i^*)=2$   if  $\rank_p(K_i)=2$ 
and  $\rank_p(K_i^*)=0$   if   $\rank_p(K_i)=0$.
Suppose $\rank_p(K_i)=2$. Then, by replacing $i$ by some $j\in J(i)$ if
necessry, we can assume that $c(i,2r)=0=c(i+r,2r)$. We shall show that
$c(i,3r)=0$ and $c(i+r,r)=0$. Since these entries occur in different
anti-diagonal $2\times 2$ blocks, this will force $\rank_p(K_i^*)=2$.
Let $i=(i_0,i_1)$ and $i+r=(j_0,j_1)$. The hypotheses mean that 
$i_0$, $i_1$, $j_0$, $j_1\leq\frac{p-1}{2}$. Since $\frac{3p-1}{4}\geq \frac{p-1}{2}\geq j_0$, when $r=(\frac{3p-1}{4},\frac{p-3}{4})$ is added to $i$ a carry must be generated 
by the addition of the first digits. This implies that
$i_0>\frac{p-3}{4}$. Since $\frac{p-3}{4}+1+\frac{p-1}{2}\leq p-1$,
no carry is generated from the addition of the second digits of $r$ and $i$.
Thus, $j_1=i_1+\frac{p-3}{4}+1$. Since $j_1\leq\frac{p-1}{2}$, we deduce that 
$i_1\leq\frac{p+1}{4}-1=\frac{p-3}{4}$. It is then easily checked that
$c(i,3r)=0$.  Also, since $j_0+p=i_0+\frac{3p-1}{4}$, we have $j_0\leq\frac{p-3}{4}$,
from which it follows that $c(i+r,r)=0$. We have proved that $\rank_p(K_i^*)=2$
if $\rank_p(K_i)=2$. 
Next suppose that $\rank_p(K_i)=0$. Since $\det(K_i)=p^4$, we must have
$c(i,2r)=c(i+r,2r)=c(i+2r,2r)=c(i+3r,2r)=1$. Since $s(2r)=p-1$,
we deduce from the formula $c(u,v)=\frac{1}{p-1}(s(u)+s(v)-s(u+v))$
that $s(i)=s(i+r)=s(i+2r)=s(i+3r)$. It then follows that all of the
nonzero entries of (\ref{vistar}) are equal to 1, so that $\rank_p(K_i^*)=0$.
\end{proof}

We are now ready to complete the proof of similarity of $K_i$ and $K_i^*$.
We have seen that $\rank_p(K_i^*)=\rank_p(K_i)\in\{0,1,2\}$. For each
$p$-rank, we exhibit a basis of $M_i$ for which the matrix
of the restriction of $\mu_{K^*}$ is the corresponding matrix in (\ref{canon}).

Suppose $\rank_p(K_i^*)=0$. Then all nonzero entries of (\ref{Kistar})
are exactly divisible by $p$.
We set $v_i=e_i$ $v_2=\frac{1}{p}(\overline\alpha J(i,r)e_{i+r}+\alpha J(i,3r)e_{i+3r})$,
$v_3=e_{i+2r}$, and  $v_4=\frac{1}{p}(\alpha J(i+2r,3r)e_{i+r}+\overline\alpha J(i+2r,r)e_{i+3r})$. It is easy to check that $v_1$, $v_2$, $v_3$ and $v_4$ form
a basis of $M_i$.
Then $\mu_ {K^*}(v_1)=pv_2$ and 
\begin{equation}
\begin{aligned}
\mu_{K^*}(v_2)&=\frac{1}{p}[\overline\alpha J(i,r)\mu_{K^*}(e_{i+r})+\alpha J(i,3r)\mu_{K^*}(e_{i+3r})]\\
&=\frac{1}{p}[\overline\alpha J(i,r)(\alpha J(i+r,3r)e_i+\overline\alpha J(i+r,r)e_{i+2r})\\
&+\alpha J(i,3r)(\overline\alpha J(i+3r,r)e_i+\alpha J(i+3r,3r)e_{i+2r})]\\
&=\frac{1}{p}\overline\alpha\alpha[J(i,r)J(i+r,3r)+J(i,3r)J(i+3r,r))]e_i\\
&+\frac{1}{p}[\overline\alpha^2J(i,r)J(i+r,r)+\alpha^2J(i,3r)J(i+3r,3r)]e_{i+2r}
\end{aligned}
\end{equation}
By Lemma~\ref{berndt}(ii), and the fact that $\alpha^2+\overline\alpha^2=0$,
we see that the coefficient of $e_{i+2r}$ is zero. Thus $v_1$ and $v_2$
span a $\mu_{K^*}$-invariant subspace on which $\mu_{K^*}$ has trace zero,
hence determinant $-p^2$. It follows that the coefficient of $v_1$
must be $p$. The same calculation with $i$ replaced by $i+2r$
shows that $\mu_{K^*}(v_3)=pv_4$ and $\mu_{K^*}(v_4)=pv_3$, so with respect to the
basis $v_1$, $v_2$, $v_3$, $v_4$, the matrix of $\mu_{K^*}$ on $M_i$
is the first matrix in (\ref{canon}).

Suppose $\rank_p(K_i^*)=1$. Then,  $D\neq 0$ as otherwise the $p$-rank would be even.
Up to replacing $i$ by some $j\in J(i)$, we can assume $D=1$. Then by a
further change of $i$ with $i+2r$ if nececessary 
 we can assume that the matrix (\ref{vistar}) of valuations is
\begin{equation}\label{rank1}
\begin{bmatrix}
0&0&2&2\\
0&0&1&1\\
0&1&0&0\\
0&1&0&0\\
\end{bmatrix},
\end{equation}
We set $v_i=e_i$ $v_2=\overline\alpha J(i,r)e_{i+r}+\alpha J(i,3r)e_{i+3r}$,
$v_3=e_{i+2r}$, and  $v_4=\frac{1}{p}(\alpha J(i+2r,3r)e_{i+r}+\overline\alpha J(i+2r,r)e_{i+3r})$. It is straightforward to check that $v_1$, $v_2$, $v_3$ and $v_4$ form a 
basis for $M_i$ and  a similar calculation to the $p$-rank $0$ case shows that on $M_i$
the matrix of $\mu_{K^*}$ with respect to this basis is the second matrix of
(\ref{canon}).

Suppose $\rank_p(K_i^*)=2$. Then, we must have $D=0$, since
neither anti-diagonal block can have $p$-rank $0$.
Up to replacing $i$ by $j\in J(i)$, we can assume that in (\ref{vistar}) 
$c(i,3r)=c(i+3r,r)=0$, whence $c(i+2r,3r)=c(i+r,r)=2$, by Lemma~\ref{abcd}(ii).
We claim that $c(i,r)=c(i+2r,r)=1$. Suppose not. Then one of $c(i,r)$
and $c(i+2r,r)$ is zero and the other is $2$. If $c(i,r)=0$ and  $c(i+2r,r)=2$, let
$i=(i_0,i_1)$. Since also $c(i,3r)=0$, we must have $i_0$, $i_1\leq\frac{p-3}{4}$.
Then $i+2r=(i_0+\frac{p-1}{2}, i_1+\frac{p-1}{2})$, with
$i_1+\frac{p-1}{2}\leq\frac{3p-5}{4}=\frac{3p-1}{4}-1$. This means that in adding $r$ to
$i+2r$, there can be no carry generated in the second digit, which contradicts
the assumption that $c(i+2r,r)=2$. If $c(i,r)=2$ and  $c(i+2r,r)=0$, 
we obtain a contradiction similarly.  Thus, we may assume that the matrix
of valuations (\ref{vistar}) is

\begin{equation}\label{rank2}
\begin{bmatrix}
0&0&1&2\\
0&0&0&1\\
1&2&0&0\\
0&1&0&0\\
\end{bmatrix},
\end{equation}

We set $v_i=e_i$ $v_2=\overline\alpha J(i,r)e_{i+r}+\alpha J(i,3r)e_{i+3r}$,
$v_3=e_{i+2r}$, and  $v_4=\alpha J(i+2r,3r)e_{i+r}+\overline\alpha J(i+2r,r)e_{i+3r}$,
and easily check that these vectors form a basis of $M_i$. Then a similar calculation
to the $p$-rank $0$ case, shows that the matrix of $\mu_{K^*}$ with respect to this basis is the third matrix of (\ref{canon}).

The proof that $K_i$ and $K_i^*$ are similar over $R$ is now complete.

\begin{remark} For $p>3$ we do not know when, if ever, $A(p^2)$ and $A^*(p^2)$ are 
similar over $\Z$. 
\end{remark}


\end{document}